\documentclass
[12pt,%
]%
{article}

\usepackage{amsthm}
\usepackage{amssymb}
\usepackage{amsfonts,amsmath,mathrsfs,amscd,latexsym}
\usepackage{verbatim}
\usepackage{eucal}
\usepackage{graphicx}

\usepackage{enumerate}

\usepackage[pdftex,unicode,colorlinks,linkcolor=blue,
citecolor=red,bookmarksopen,pdfhighlight=/N]{hyperref}

\newtheorem{theorem}{Theorem}
\newtheorem{propos}{Proposition}

\theoremstyle{definition}
\newtheorem{definition}{Definition}
\newtheorem{example}{Example}
\newtheorem{remark}{Remark} 
\newcommand{\RR}{\mathbb{R}} 

\newcommand{\NN}{\mathbb{N}}

\newcommand{\BB}{\mathbb{B}}

\newcommand{\dd}{\,{\rm d}}

\DeclareMathOperator{\infill}{\rm in-fill}

\DeclareMathOperator{\clos}{clos} 
\DeclareMathOperator{\Int}{int}
\DeclareMathOperator{\Meas}{\sf Meas}

\DeclareMathOperator{\har}{\sf har}

\DeclareMathOperator{\comp}{cmp}

\DeclareMathOperator{\sbh}{\sf sbh}

\DeclareMathOperator{\supp}{supp}

\DeclareMathOperator{\sgn}{sgn}

\DeclareMathOperator{\Borel}{\sf Bor}

\DeclareMathOperator{\Conn}{Conn}

\begin{document}

\title{Balayage of Measures with respect to
Classes\\ of Subharmonic 
and Harmonic Functions}

\author{B.\,N. Khabibullin}





\maketitle

\begin{abstract}
We investigate some properties of balayage, or, sweeping (out), of measures with respect to subclasses of subharmonic functions.  The following issues are considered: relationships between balayage of measures with respect to classes of harmonic or subharmonic functions and  balayage of measures with respect to significantly smaller classes of specific classes of functions; integration of measures and balayage of measures; sensitivity of balayage  of measures to polar sets, etc.
\end{abstract}

\section{Introduction}


The origins of the concept of {\it balayage,\/} or,  ``sweeping (out)'' etc., of measures or functions are the studies of Henri Poincar\'e, de la Vall\'ee Poussin, Henri Cartan, Marcel Brelot and many others. 
A detailed historical review of potential theory is given in  in  \cite{Brelot10}. 
In \cite{KhaRozKha19}, we investigate  various general concepts of  balayage. In this article we deal with particular cases of such balayage with respect to special  classes of  subharmonic  functions. 

The general concept of balayage can be defined as follows. 
Let $R$ be a   (pre-)ordered set with a (pre-)order relation $\leq$.
Let $L$ be a set with a subset $H\subset L$. 
A function $\omega\colon L\to R$ can be called the {\it  balayage} of a function 
 $\delta \colon L\to R$ {\it  with respect to\/} $H$, and we write 
$\delta\preceq_H  \omega$,  if the function $\omega$ majorizes the function $\delta$ on $H$:
\begin{equation}\label{b0}
\delta(h)\leq \omega(h) \quad \text{for each $h\in  H\subset L$} .
\end{equation}
In this article, $R$ is the extended real line,  $L$ is the class of all upper semicontinuous functions  on an open set $O$ in a finite-dimensional Euclidean space, $H$ is a subclass of subharmonic functions on $O$, and 
$\delta$ and $\omega$ is a pair of Radon positive measures on $O$ with compact supports in $O$. In this case, relationship \eqref{b0} turns into inequalities of the form 
\begin{equation}\label{b0mu}
\delta(h):=\int_O h \dd \delta \leq \int_O h \dd \omega=:\omega(h) \quad \text{for each  $h\in H\subset L$}. 
\end{equation}

We investigate properties of balayage of  measures with respect to classes of harmonic, subharmonic, and special subharmonic functions. 


We proceed to precise and detailed definitions and formulations.

\section{Definitions, notations and conventions}\label{Ss12}

The reader can skip this Section \ref{Ss12}
and return to it only if necessary.

 We denote by $\NN:=\{1,2,\dots\}$, $\RR$, and $\RR^+:=\{x\in \RR\colon x\geq 0\}$  the sets of {\it natural,\/} of {\it real,\/} and  of {\it positive\/} numbers, each endowed with its natural order ($\leq$, $\sup/\inf$), algebraic, geometric  and topological structure.  We denote singleton sets by a symbol without curly brackets. So, $\NN_0:=\{0\}\cup \NN=:0\cup \NN$, and  $\RR^+\!\setminus\!0:=\RR^+\!\setminus\!\{0\}$ is the set of {\it strictly positive\/} numbers, etc. 

The {\it extended real line\/}  $\overline \RR:=-\infty\sqcup\RR\sqcup+\infty$ is the order completion of $\RR$ by the  {\it disjoint union\/} $\sqcup$  with $+\infty:=\sup \RR$ and $-\infty:=\inf \RR$ equipped with the order topology with two  ends $\pm\infty$, $\overline \RR^+:=\RR^+\sqcup+\infty$;  $\inf \varnothing :=+\infty$, $\sup \varnothing :=-\infty$ for the {\it empty set\/} $\varnothing$ etc. 
The same symbol $0$ is also used, depending on the context, to denote  zero vector, zero function, zero measure, etc.

We denote by $\RR^{\tt d}$ the  {\it Euclidean space of ${\tt d}\in \NN$ dimensions\/}  with the  {\it Euclidean norm\/} $|x|:=\sqrt{x_1^2+\dots+x_{\tt d}^2}$ of $x=(x_1,\dots ,x_{\tt d})\in \RR^{\tt d}$.

We denote by $\RR^{\tt d}_{\infty}:=\RR^{\tt d}\sqcup\infty$  the {\it  Alexandroff\/}  {\it one-point compactification\/}   of $\RR^{\tt d}$
obtained by adding one extra point $\infty$. For a subset $S\subset \RR^{\tt d}_{\infty}$ or a subset $S\subset \RR^{\tt d}$ we let $\complement S :=\RR^{\tt d}_{\infty}\!\setminus\!S$, $\clos S$, $\Int S:=\complement (\clos \complement  S)$, and $\partial S:=\clos S\!\setminus\!\Int S$ denote its
 {\it complement,\/} {\it closure,} {\it interior,} and {\it boundary\/}  always in $\RR^{\tt d}_{\infty}$, and $S$ is equipped with the topology induced from $\RR^{\tt d}_{\infty}$. If $S'$ is a relative compact subset in $S$, i.e., $\clos S'\subset S$,  then we write $S'\Subset S$.  We denote by 
 $B(x,t):=\{y\in \RR^{\tt d}\colon |y-x|< t\}$, $\overline B(x,t):=\{y\in \RR^{\tt d}\colon |y-x|\leq  t\}$, $\partial \overline B(x,t):=\overline B(x,t)\!\setminus\!  B(x,t)$  an {\it open ball, closed ball,\/} a {\it circle of radius $t\in \RR^+$ centered at $x\in \RR^{\tt d}$}, respectively. Besides, we denote by  $\BB:=B(0,1)$, $\overline \BB:=\overline B(0,1)$ and $\partial \BB:=\partial \overline B(0,1)$  the {\it open unit ball,\/} the {\it closed unit ball\/} and the {\it unit sphere\/} in $\RR^{\tt d}$, respectively.

\underline{Throughout this paper} $O\neq \varnothing$ will denote  an  {\it open subset  in\/ $\RR^{\tt d}$},   and $D\neq \varnothing$ is a  {\it domain in\/ $\RR^{\tt d}$,\/}  i.e., an open connected subset in $\RR^{\tt d}$. 

For $S\subset \RR_{\infty}^{\tt d}$, $C(S)$  is the vector space over $\RR$ of all {\it continuous\/}  functions $f\colon S\to  \RR$ with the $\sup$-norm, 
$C_0(S)\subset C(S)$ is the subspace of functions $f\in C(S)$ {\it with
compact support\/} $\supp f\Subset S$,  
and  $\text{\sf usc}(S)$ is the convex cone over $\RR^+$ of all {\it upper semicontinuous\/} functions $f\colon S\to \RR\cup -\infty=\overline\RR\!\setminus\!+\infty$. For $S\subset \RR^{\tt d}$, 
$\har(S)$ and  $\sbh(S)$ are the collections   of  all functions $u$ which are harmonic and subharmonic on some open set $O_u\supset S$, respectively. 
In addition, $\sbh_*(S)\subset \sbh(S)$ consists only of functions $u\in \sbh(S)$ such that $u\not\equiv -\infty$ on each connected component of $O_u$.

The convex cone over $\RR^+$ of all Borel, or Radon,  positive measures $\mu\geq 0$  on the $\sigma$-algebra $\Borel (S)$ of all {\it Borel subsets\/} of $S$ is denoted by $\Meas^+(S)$; $\Meas^+_{\comp}(S)\subset \Meas^+(S)$ is the subcone of measures $\mu\in \Meas^+(S)$ with compact  {\it support\/} $\supp \mu$ in $S$, $\Meas^{+1}(S)
$ is the convex set of {\it probability\/} measures on $S$, 
$\Meas_{\comp}^{1+}(S):=\Meas^{1+}(S)\cap \Meas_{\comp}(S)$.
So, $\delta_x \in \Meas_{\comp}^{1+} (S)$
is the {\it Dirac measure\/} at a point $x \in S$, i.e., $\supp \delta_x = \{x\}$, $\delta_x (\{x\}) = 1$. 

We denote by  $\mu\bigm|_{S'}$
the restriction of $\mu$ to  $S'\in {\Borel} (S)$. The same notation is used for the restrictions of functions and their classes to sets.


Let ${\bigtriangleup}$  be the {\it Laplace operator\/}  acting in the sense of the
theory of distributions, $\Gamma$ be the \textit{gamma function}.
For $u\in \sbh_*(O)$, the {\it  Riesz measure of\/} $u$ is a  Borel 
(or Radon \cite[A.3]{R}) \textit{ positive measure }
\begin{equation}\label{df:cm}
\varDelta_u:= c_d {\bigtriangleup}  u\in \Meas^+(  O),  \quad 
c_d:=\frac{\Gamma(d/2)}{2\pi^{d/2}\max \{1, d-2\bigr\}}.
\end{equation}

\section{Inward filling of subsets in an  open set}\label{hullin} 

Let $O$ be a topological space, and $S\subset O$.  
We denote by $\Conn_O S$ the set of all connected components of $S$. 
We write $S\Subset O$, if the closure of $S$ in $O$ is a compact subset of $O$.  
\begin{definition}
\label{df:hole}
The union of a subset $S\subset O$ with all 
connected component of $C\in \Conn_O( O\!\setminus\!S)$ such that $C\Subset O$
 will be called the  \textit{inward filling\/} of $S$ with respect to $O$ and is denoted further as 
\begin{equation*}
\infill_O S:=S\bigcup \Bigl(\bigcup \bigl\{C\in \Conn_O (O\!\setminus\!S) \colon C\Subset O\bigr\}\Bigr). 
\end{equation*}
Denote by $O_{\infty}$  the {\it  Alexandroff one-point compactification of\/} $O$ 
with underlying set $O \sqcup \{\infty\}$.
\end{definition} 

\begin{propos}[{\rm \cite[6.3]{Gardiner}, \cite{GautherB}}]\label{KOc}
Let  $S$ be a compact  set in an open set  $O \subset \RR^{\tt d}$.  Then 
\begin{enumerate}[{\rm (i)}]
\item\label{Ki} $\infill_{O} S$ is a compact subset in $O$, and 
$\infill_{O}\bigl(\infill_{O}  S\bigr)=\infill_{O}  S$;
\item\label{Kii} the set\/ $O_{\infty} \!\setminus\!\infill_{O}  S$ is connected and locally connected; 
\item\label{Kiii}   the inward filling of $S$ with respect to $O$ coincides with the complement in $O_{\infty}$ of  connected component of $O_{\infty}\!\setminus\!S$ containing the point $\infty$;

\item\label{Kiiv} if $O'\subset \RR_{\infty}^{\tt d}$ is an open subset and 
 $O\subset  O'$, then  $\infill_{O}  S\subset \infill_{O'}  S$;
\item\label{Kiv} $\RR^{\tt d}\!\setminus\!\infill_{O}  S$ has only finitely many components, i.\,e., $\#\Conn_{\RR^{\tt d}_{\infty}}(\RR^{\tt d}\!\setminus\!\infill_{O}  S)<\infty$.
\end{enumerate}
\end{propos}

\begin{propos}[{\cite[Theorem 1.7]{Gardiner}}]\label{pr2}
Let $O$ be an open set in $\RR^{\tt d}$, let $S$ be a compact subset in $O$, and suppose that $O_{\infty}\!\setminus\!S$ is connected. Then, for each $u\in \har (S)$ and each number $b\in \RR^+\!\setminus\!0$, there is $h\in \har(O)$ such that $|u-h|<b$ on $S$.  
\end{propos}

\begin{propos}\label{pr3}
Let $O$ be an open set in $\RR^{\tt d}$, and let $S$ be a compact subset in $O$. If $h\in \har\bigl(\infill_O S\bigr)$, then 
 there are harmonic functions $ h_j\underset{\text{\tiny $j\in \NN$}}{\in} \har (O)$  such that the sequence  $(h_j)_{j\in \NN}$ converges to this harmonic  
function $h$ in $C (\infill_O S)$.
\end{propos}
 Proposition \ref{pr3} is the intersection of Proposition \ref{KOc} (parts \eqref{Ki}--\eqref{Kii}) and Proposition \ref{pr2} if we consider $\infill_O S$ instead of $S$ in Proposition \ref{pr2}.

\begin{propos}[{\cite[Theorem 6.1]{Gardiner},
\cite[Theorem 1]{GautherC}, \cite[Theorem 16]{GautherB}}]\label{prs}
Let $O$ be an open set in $\RR^{\tt d}$,  let $S$ be a closed subset in $O$, and suppose that  $O_{\infty}\!\setminus\!S$ is connected and locally connected. Then, for each $u\in \sbh (S)$,  there exists $U\in \sbh(O)$ such that $u=U$ on $S$.  
\end{propos}

The intersection of Proposition \ref{KOc} (parts \eqref{Ki}--\eqref{Kii}) and Proposition \ref{prs} is 
\begin{propos}\label{prs+}
Let $O$ be an open set in $\RR^{\tt d}$,  and let $S$ be a compact subset in $O$. Then, for each $u\in \sbh (\infill_O S)$,  there exists $U\in \sbh(O)$ such that 
$u=U$ on $\infill_O S$.  
\end{propos}

\section{Balayage of measures}\label{Ssec_balm}

In this section \ref{Ssec_balm} we traditional    classical balayage that is particular case of  \eqref{b0} (see also \cite{KhaRozKha19}).

\begin{definition}[{\cite{Meyer}, \cite{BH}, \cite[Definition 5.2]{KhaRozKha19}}]\label{df:1} Let $S\subset \Borel (\RR^{\tt d}_{\infty})$, $\delta\in \Meas^+_{\comp}(S)$, $\omega  \in \Meas^+_{\comp}(S)$.   Let 
$H\subset \text{\sf usc}(S)$ be a subclass of upper semicontinuous  functions on $S$. 
  We write ${\delta} \preceq_H \omega$ and say that the measure  $\omega$ is a {\it balayage,\/} or, sweeping (out), of the measure ${\delta}$ {\it with respect to\/} $H$, or, briefly, $\omega$ is  $H$-balayage of $\delta$,   if 
\begin{equation}\label{balnumu}
\int h \dd {\delta} \overset{\eqref{b0mu}}{\leq} \int h\dd \omega \quad\text{\it for each\/ $h\in H$.}
\end{equation} 
If $\delta\preceq_H \omega$ and at the same time $\omega\preceq_H\delta$, then we write $\delta\simeq_H \omega$. 
\end{definition}

\begin{propos}\label{Prtr}  Let  $O\subset \RR^{\tt d}$ be an open set, $\omega \in \Meas(O)$ be a  $H$-balayage of ${\delta}\in  \Meas(O)$, $O'\subset \RR^{\tt d}$ be an open set, and $H'\subset{\overline \RR}^{O'}$. 
\begin{enumerate}[{\rm (i)}]
\item\label{i} The binary relation $\preceq_{H}$ (respectively $\simeq_{H}$) on $\Meas^+_{\comp}(S)$ 
 is  a  {\it preorder,\/} i.e., a reflexive and transitive relation, 
(respectively, an {\it equivalence\/}) on  $\Meas^+_{\comp}(S)$.
\item If $H$ contains a strictly positive (respectively, negative) constant, then $\delta (S)\leq \omega(S)$ (respectively, $\delta (S)\geq \omega(S)$). 

\item\label{b3} If $H'\subset H$, then  $\omega$  is  $H'$-balayage of ${\delta}$.
 \item If  $O'\subset O$  and $\supp \delta\cup \supp \omega\subset O'$, then $\omega\bigm|_{O'}$  is a balayage of ${\delta}\bigm|_{O'}$ for  $H\bigm|_{O'}$.

\item\label{pm} If $H=-H$, then the order  $\preceq_{H}$ is the equivalence $\simeq_{H}$. 
So, if $H=\har(S)$, then   $\omega $ is a $\har(S)$-balayage of 
$\delta $ if and only if $\delta\simeq_{\har(S)}\omega$, i.e.,
\begin{equation}\label{bhar}
\int_S h\dd \delta =\int_S h\dd \omega \quad\text{\it for each $h\in \har(S)$} \quad\text{\it and} \quad \delta(S)=\omega(S). 
\end{equation}

\item\label{4} If $\delta \preceq_{\sbh(S)}\omega$, then $\delta \preceq_{\har(S)}\omega$. The converse is not true\/
{\rm \cite[XIB2]{Koosis}, \cite[Example]{MenKha19}.} 
\item\label{pr:diff} If $\omega \in \Meas^+_{\comp}(O)$   is a 
$\bigl(\sbh(O)\cap C^{\infty}(O)\bigr)$-balayage of $\delta\in \Meas_{\comp}^+(O)$, where $C^{\infty}(O)$ is the class 
of all infinitely differentiable functions on $O$, then  $\delta\preceq_{\sbh(O)}\omega$, since for each function $u\in \sbh(O)$ there exists a sequence of functions $u_j\underset{\text{\tiny $j\in \NN$}}{\in} \sbh(O)\cap C^{\infty}(O)$ decreasing to it\/ {\rm \cite[Ch. 4, 10, Approximation Theorem]{Doob}.}

 \end{enumerate}
\end{propos}
All statements of Proposition \ref{Prtr} are obvious.

\begin{example}[{\rm \cite{Gamelin}, \cite{C-R}, \cite{Ransford01}, \cite{HN11}}]\label{sbhJ} Let $x\in O$. 
If a measure $\omega\in \Meas_{\comp}^+(O)$  is  a balayage of the Dirac measure $\delta_{x}$ with respect to $\sbh(O)$, then this measure $\omega$ is called a {\it Jensen measure on $O$ at\/} $x$. The class of all Jensen measures on $O$ at $x\in O$ will be denoted by $J_x(O)$. 
\end{example}

\begin{example}[{\rm \cite{Gamelin}, 
\cite[3]{Gamelin}, \cite{Kha03}, \cite[Definition 8]{Kha07}}]\label{sbhAS} Let $x\in O$.   If $\omega\in \Meas_{\comp}^+(O)$ is  $\har(O)$-balayage of the Dirac measure $\delta_{x}$, then the measure $\omega$ is called an {\it Arens\,--\,Singer  measure on $O$ at\/} $x\in O$. The class of all  Arens\,--\,Singer measures on $O$ at $x$  is denoted by $AS_x(O)\supset J_x(O)$. 
\end{example}

For $s\in \RR$, we set  
\begin{subequations}\label{kK}
\begin{align}
k_s(t)& := \begin{cases}
\ln t  &\text{ if $s=0$},\\
 -\sgn (s)  t^{-s} &\text{ if $s\in \RR\!\setminus\!0$,} 
\end{cases}
\qquad  t\in \RR^+\!\setminus\!0,
\tag{\ref{kK}k}\label{{kK}k}
\\
K_{d-2}(y,x)&:=\begin{cases}
k_{d-2}\bigl(|y-x|\bigr)  &\text{ if $y\neq x$},\\
 -\infty &\text{ if $y=x$ and $d\geq 2$},\\
0 &\text{ if $y=x$ and  $d=1$},\\
\end{cases}
\quad  (y,x) \in \RR^{\tt d}\times \RR^{\tt d},
\tag{\ref{kK}K}\label{{kK}K}\\
{\sf k}_x&\colon y\underset{\text{\tiny $y\in \RR^{\tt d}$}}{\longmapsto} 
K_{d-2}(y,x) \; \in  \sbh(\RR^{\tt d})\bigcap \har(\RR^{\tt d}\!\setminus\!x), 
\quad x\in \RR^{\tt d}, 
\tag{\ref{kK}${\tt k}_x$}\label{kh}
\\
{\sf K}(X)&:=\{{\sf k}_x\colon x\in X\}\subset \sbh_*(\RR^{\tt d}), \quad X\subset \RR^{\tt d}.  
\tag{\ref{kK}{\tt \underline{K}}}\label{Kcl}
\end{align}
\end{subequations}

\begin{theorem}\label{th1} Let $O\subset \RR^{\tt d}$ be an open set, and $\delta \in \Meas^+_{\comp}(O)$,  
$\omega \in \Meas^+_{\comp}(O)$. 
The measure $\omega$ is $\har(O)$-balayage (respectively, $\sbh(O)$-balayage) of the measure $\delta$ 
if and only if  there exists a compact subset $S\Subset O$ such that  this measure  $\omega$ is a balayage  of  $\delta$ with respect to 
\begin{subequations}\label{KO}
\begin{align}
{\sf K}(O\!\setminus\!S)&\bigcup
\bigl(-\text{\sf K}(O\!\setminus\!S)\bigr),
\tag{\ref{KO}h}\label{{KO}h}
\\
\Bigl(\text{respectively, } 
{\sf K}(O)&\bigcup
\bigl(-{\sf K}(O\!\setminus\!S)\bigr)\Bigr).
\tag{\ref{KO}s}\label{{KO}s}
\end{align}
\end{subequations}
\end{theorem}

\begin{proof}  We set 
\begin{equation}\label{SO}
S_O:=\infill_{O}(\supp \delta \cup \supp \omega).
\end{equation}
By Proposition \ref{pr3},  for each $x\notin S_O$ there are  functions $\pm h_j^x\underset{\text{\tiny $j\in \NN$}}{\in} \har (O)$  such that the sequence  $(\pm h_j^x)_{j\in \NN}$ converges to $\pm{\sf k}_x\subset \har(S_O)$ in $C ( S_O)$. Let 
\begin{equation}\label{bald}
\Bigl(\delta\preceq_{\har (O)}\omega\Bigr)\Longleftrightarrow 
\Bigl(\delta\simeq_{\har (O)}\omega\Bigr)
\quad \text{(see Definition \ref{df:1} and Proposition \ref{Prtr}\eqref{i},\eqref{pm})}.
\end{equation}
If $x\notin S_O$, then,    
\begin{multline*}
\int \pm K_{d-2}(y,x)\dd \delta(y)\overset{\eqref{SO}}{=}
\int_{S_O}\pm K_{d-2}(y,x)\dd \delta(y)=
 \int_{S_O} \lim_{j\to \infty}\pm h_j^x(y) \dd {\delta}(y)=
\lim_{j\to \infty} \int_{S_O} \pm h_j^x \dd {\delta}
\\
\overset{\eqref{SO}}{=} \lim_{j\to \infty} \int_{O} \pm h_j^x \dd {\delta}
\overset{\eqref{bald},\eqref{bhar}}{=} 
\lim_{j\to \infty} \int_{O} \pm h_j^x \dd {\omega}\\
\overset{\eqref{SO}}{=} 
\lim_{j\to \infty} \int_{S_O} \pm h_j^x \dd {\omega}=
\int_{S_O}  \lim_{j\to \infty}
\pm h_j^x \dd {\omega} =
\int_{S_O}\pm K_{d-2}(y,x)\dd \omega(y)
\overset{\eqref{SO}}{=} \int \pm  K_{d-2}(y,x)\dd \omega(y).
\end{multline*}
Thus,  \eqref{bald} implies  that    
$\omega$ is a balayage of   $\delta$ with respect to the class \eqref{{KO}h} with $S:=S_O$. 


If $\delta\preceq_{\sbh(O)}\omega$, then, by Proposition \ref{Prtr}\eqref{4}, 
$\delta\preceq_{\har(O)}\omega$, and
$\delta\preceq_{{\sf K}(\RR^{\tt d}\!\setminus\!S_O)\cup(-{\sf K}(\RR^{\tt d}\!\setminus\!S_O))}\omega$. Besides,
 in view of \eqref{kh},  we obtain
\begin{multline*}
\int K_{d-2}(y,x)\dd \delta(y)\overset{\eqref{kh}}{=}
\int_O {\sf k}_x(y)\dd \delta(y)\\
\overset{\eqref{balnumu}}{\leq}\int_O {\sf k}_x(y)\dd \delta(y)\overset{\eqref{kh}}{=}
 \int K_{d-2}(y,x)\dd \omega(y)\quad\text{for each $x\in \RR^{\tt d}$.}
\end{multline*}
Thus,  $\delta\preceq_{\sbh(O)}\omega$ implies $\delta\preceq_{{\sf K}(\RR^{\tt d})}\omega$ and $\omega$ is a balayage of   $\delta$ with respect to \eqref{{KO}s} if $S:=S_O$. 

So, the necessary conditions of Theorem \ref{th1} are proved.

In the opposite direction, let 
\begin{equation}\label{sT}
\delta\overset{\eqref{{KO}h}}{\preceq}_{{\sf K}(O\!\setminus\!S)\cup(-{\sf K}(O\!\setminus\!S))} \omega, \quad\text{where $S\overset{\text{\tiny closed}}{=}\clos S
\overset{\text{\tiny compact}}{\Subset}O $} .
\end{equation}
Then, by Definition \ref{df:1} and Proposition \ref{Prtr}\eqref{pm}, according to equality \eqref{bhar}, we have 
\begin{equation}\label{K=}
\int_S K_{d-2}(y,x)\dd \delta(y)=\int_S K_{d-2}(y,x)\dd \omega (y)\quad \text{for each $x\in O\!\setminus\!S$.}
\end{equation}
Let $u\in \har (O)$. Without loss of generality, we can assume that
\begin{equation}\label{Sdo}
\supp \delta \bigcup \supp \omega \Subset \Int S\subset 
S\overset{\text{\tiny closed}}{=}\clos S
\overset{\text{\tiny compact}}{\Subset}O.
\end{equation}
There is an open subset $U\Subset  O$ such that $S\subset U$, $\partial U$ is a $C^1$ surface and that each point of $\partial U$ is a one-sided boundary point of  $U$ \cite[1.6]{Gardiner}. In particular $\partial U\subset O\!\setminus\! S$. If we apply Green's identity to  $U\!\setminus\!\overline B(x,r)$ and let $r$ tend to $0$, we obtain  \cite[1.6]{Gardiner}
\begin{equation}\label{uU}
u(y)=c_d\int_{\partial U}\Bigl(K_{d-2}(y,x) \frac{\partial u}{\partial \vec{n}_x }(x)
-u(x) \frac{\partial }{\partial \vec{n}_x}K_{d-2}(y,x)\Bigr)
 \dd \sigma (x) \quad\text{ for each $y\in S$},   
\end{equation}
where $\sigma$ denotes surface area measure on $\partial U$,  $\vec{n}_x$ denotes the outer unit normal to $\partial U$ at $x\in \partial U$ and   $c_d\in \RR^+\!\setminus\!0$ is defined in \eqref{df:cm}. Integrating both sides of equality \eqref{uU} with respect to the measure $\delta$ and the measure $\omega$, we obtain, respectively,
\begin{subequations}\label{uUid}
\begin{align}
\begin{split}
\frac{1}{c_d}\int_{\supp \delta} u(y)\dd \delta(y)&\overset{\eqref{Sdo}}{=}\int_S\int_{\partial U} K_{d-2}(y,x) \frac{\partial u}{\partial \vec{n}_x }(x)\dd \sigma (x)\dd \delta(y)\\
&-
\int_S\int_{\partial U}u(x) \frac{\partial }{\partial \vec{n}_x}K_{d-2}(y,x)
 \dd \sigma (x)\dd \delta(y),
\end{split}
\tag{\ref{uUid}$\delta$}\label{{uUid}d}
\\
\begin{split}
\frac{1}{c_d}\int_{\supp \delta} u(y)\dd \omega(y)&\overset{\eqref{Sdo}}{=}\int_S\int_{\partial U} K_{d-2}(y,x) \frac{\partial u}{\partial \vec{n}_x }(x)\dd \sigma (x)\dd \omega(y)\\
&-
\int_S\int_{\partial U}u(x) \frac{\partial }{\partial \vec{n}_x}K_{d-2}(y,x)
 \dd \sigma (x)\dd \omega(y).   
\end{split}
\tag{\ref{uUid}$\omega$}\label{{uUid}o}
\end{align}
\end{subequations}
Hence, using  Fubini's theorem and differentiation under the integral sign, we have
\begin{subequations}\label{uUid+}
\begin{align}
\begin{split}
\frac{1}{c_d}\int_{\supp \delta} u(y)\dd \delta(y)&\overset{\eqref{{uUid}d}}{=}\int_{\partial U} \biggl(\int_S K_{d-2}(y,x)\dd \delta(y) \biggr)\frac{\partial u}{\partial \vec{n}_x }(x)\dd \sigma (x)\\
&-
\int_{\partial U}u(x) \frac{\partial }{\partial \vec{n}_x}\biggl(\int_S K_{d-2}(y,x)\dd \delta(y) \biggr)
 \dd \sigma (x),
\end{split}
\tag{\ref{uUid+}$\delta$}\label{{uUid}d+}
\\
\begin{split}
\frac{1}{c_d}\int_{\supp \delta} u(y)\dd \omega(y)&\overset{\eqref{{uUid}o}}{=}\int_{\partial U} \biggl(\int_S K_{d-2}(y,x)\dd \omega(y) \biggr) \frac{\partial u}{\partial \vec{n}_x }(x)\dd \sigma (x)\\
&-
\int_{\partial U}u(x) \frac{\partial }{\partial \vec{n}_x}\biggl(\int_S K_{d-2}(y,x)\dd \omega(y) \biggr)
 \dd \sigma (x).   
\end{split}
\tag{\ref{uUid+}$\omega$}\label{{uUid}o+}
\end{align}
\end{subequations}
According to equality \eqref{K=}, for each $x\in \partial U\subset O\!\setminus\!S$, the internal integrals on the right-hand sides of equalities \eqref{{uUid}d+} and \eqref{{uUid}o+} coincide, and the external integrals on the right-hand sides of equalities  \eqref{{uUid}d+} and \eqref{{uUid}o+}  are of the same form.
Therefore, the integrals on the left-hand sides of equalities \eqref{{uUid}d+} and \eqref{{uUid}o+} also  coincide for each harmonic function $u\in \har(O)$. By Definition \ref{df:1}, formula \eqref{balnumu}, this means that the measure $\omega$ is $\har(O)$-balayage  of the measure $\delta$, i.e., we have \eqref{bald}.   

It remains to consider the case when $\omega$ is a balayage of $\delta$ with respect to the class \eqref{{KO}s}. It has already been shown above that in this case 
we have \eqref{bald}, i.e., $\delta\preceq_{\har (O)}\omega$. 

Let $u\in \sbh_*(O)$ with the Riesz measure $\varDelta_u\overset{\eqref{df:cm}}{\in} \Meas^+(O)$. By the Riesz Decomposition Theorem   \cite[Theorem 3.7.1]{R}, \cite[Theorem 3.9]{HK}, \cite[Theorem 4.4.1]{AG}, \cite[Theorem 6.18]{Helms},  there exist an open set $O'\Subset O$ and a harmonic functions $h\in \har (O')$ such that $S_O\overset{\eqref{SO}}{\Subset} O'$ and 
\begin{subequations}\label{RD}
\begin{align}
u(y)&=\int_{\clos O'} K_{d-2}(x,y)\dd\varDelta_u(x)+h(y)\quad\text{for each $y\in S_O\Subset O'$},
\tag{\ref{RD}r}\label{{RD}r}\\
S&:= \supp \delta\cup \supp \omega, \quad S_O\overset{\eqref{SO}}{=}\infill S\Subset O',
\tag{\ref{RD}S}\label{{RD}S}
\end{align}
\end{subequations}
Integrating the representation \eqref{{RD}r} with respect to the measures $\delta$ and $\omega$, we obtain 
\begin{subequations}\label{RDi}
\begin{align}
\int_S u(y)\dd \delta(y)&=\int_S\int_{\clos O'} K_{d-2}(x,y)\dd\varDelta_u(x)\dd \delta(y)+\int_S h(y)\dd \delta(y),
\tag{\ref{RDi}$\delta$}\label{{RDi}r}\\
\int_S u(y)\dd \omega(y)&=\int_S\int_{\clos O'} K_{d-2}(x,y)\dd\varDelta_u(x)\dd \omega(y)+\int_S h(y)\dd \omega(y).
\tag{\ref{RDi}$\omega$}\label{{RDi}S}
\end{align}
\end{subequations}
Hence, by Fubini's theorem and in view of the symmetry of the kernel $K_{d-2}$ from \eqref{{kK}K}, we can rewrite \eqref{RDi} in the form
\begin{subequations}\label{RDiF}
\begin{align}
\int_S u(y)\dd \delta(y)&\overset{\eqref{{RDi}r}}{=}\int_{\clos O'}
\biggl(\int_S K_{d-2}(y,x)\dd \delta(y)\biggr)\dd\varDelta_u(x)+\int_S h(y)\dd \delta(y),
\tag{\ref{RDiF}$\delta$}\label{{RDi}rF}\\
\int_S u(y)\dd \omega(y)&\overset{\eqref{{RDi}r}}{=}\int_{\clos O'}
\biggl(\int_S K_{d-2}(y,x)\dd \omega(y)\biggr)\dd\varDelta_u(x)+\int_S h(y)\dd \omega(y).
\tag{\ref{RDiF}$\omega$}\label{{RDi}SF}
\end{align}
\end{subequations}
By Proposition \ref{pr3} there are harmonic functions $ h_j\underset{\text{\tiny $j\in \NN$}}{\in} \har (O)$  such that the sequence  $(h_j)_{j\in \NN}$ converges to this harmonic  
function $h\overset{\eqref{{RD}S}}{\in} \har(S_O)$ in $C ( S_O)$. Hence, 
\begin{multline}\label{ihSO}
\int_{S} h \dd {\delta}\overset{\eqref{SO}}{=}\int_{S_O} h \dd {\delta}= \int_{S_O}  \lim_{j\to \infty}
h_j \dd {\delta}  = \lim_{j\to \infty} \int_{S_O} h_j \dd {\delta}= \lim_{j\to \infty} \int_{O} h_j \dd {\delta}
\\\overset{\eqref{bald},\eqref{bhar}}{=}\lim_{j\to \infty} \int_{O} h_j \dd \omega= \lim_{j\to \infty} \int_{S_O} h_j \dd \omega= 
 \int_{S_O} \lim_{j\to \infty} h_j \dd \omega=
 \int_{S_O} h \dd {\omega}\overset{\eqref{SO}}{=} \int_{S} h \dd {\omega}.
\end{multline}

By construction of class  \eqref{{KO}s}, we have $\delta\preceq_{K(O)} \omega$. Therefore, 
\begin{equation}\label{intK}
\int_S K_{d-2}(y,x)\dd \delta(y)\leq \int_S K_{d-2}(y,x)\dd \omega(y)
\quad \text{for each $y\in O\supset \clos O'$.}
\end{equation} 

According to equality \eqref{ihSO},
the last integrals on the right-hand sides of equalities \eqref{{RDi}rF} and \eqref{{RDi}SF} also coincide, and, in view of \eqref{intK}, we have   
\begin{equation*}
\int_{\clos O'}
\biggl(\int_S K_{d-2}(y,x)\dd \delta(y)\biggr)\dd\varDelta_u(x)
\overset{\eqref{intK}}{\leq}
\int_{\clos O'}
\biggl(\int_S K_{d-2}(y,x)\dd \omega(y)\biggr)\dd\varDelta_u(x)
\end{equation*}
Hence, by representations  \eqref{{RDi}rF} and \eqref{{RDi}SF}, we obtain 
\begin{equation*}
\int_O u(y)\dd \delta(y)\overset{\eqref{{RD}S}}{=}\int_S u(y)\dd \delta(y)\overset{\eqref{RDiF}}{\leq} 
\int_S u(y)\dd \omega(y)\overset{\eqref{{RD}S}}{=}\int_O u(y)\dd \omega(y). 
\end{equation*}
 The latter, by Definition \ref{df:1}, formula \eqref{balnumu}, means that  the measure $\omega$ is $\sbh(O)$-balayage  of $\delta$. 
\end{proof}

\section{Integration of measures and balayage}

Let $S\in \Borel (O)$. Consider a function $\varTheta\colon S\to \Meas^+_{\comp}(O)$
such that 
\begin{subequations}\label{sxi}
\begin{align}
\varTheta&\colon S\to \Meas^+_{\comp}(O), \quad
\vartheta_x:=\varTheta(x), \quad 
\bigcup_{x\in S}  \supp \vartheta_x \Subset O, \quad 
\sup_{x\in S} \vartheta_x(O)<+\infty,
\tag{\ref{sxi}$\vartheta$}\label{{sxi}t}
\\
x& \underset{\text{\tiny $x\in S$}}{\longmapsto} \int_Of\dd \vartheta_x
\quad\text{is  a {\it Borel measurable\/} function for each $f\in C_0(O)$.}
\tag{\ref{sxi}B}\label{{sxi}B}
\end{align}
\end{subequations} 
  Let 
\begin{equation}\label{omega}
\omega \in \Meas_{\comp}^+(O), \quad \supp \omega\subset S\Subset O. 
\end{equation} 
Under these conditions,  we can to define the {\it integral $\int \varTheta \dd \omega$ of $\varTheta$ with respect to measure $\omega$} as a Borel, or, Radon, positive measure on $O$ \cite[Introduction, \S~1]{Landkoff}, \cite[Ch.~V, \S~3]{Bourbaki}, \cite[\S~5]{KhaShm19}
\begin{subequations}\label{iimu}
\begin{align}
\int \varTheta \dd \omega&\overset{\eqref{sxi}\text{-}\eqref{omega}}{:=:}\int_S \vartheta_x\dd\omega(x)\in \Meas^+_{\comp}(O),
\tag{\ref{iimu}I}\label{{iimu}I}
\\
\Bigl(\int \varTheta \dd \omega\Bigr)(B)&:=:\int_S \vartheta_x(B)\dd\omega(x)\in \RR
\quad\text{for each $B\in \Borel (O)$ such that $B\Subset O$,}
\tag{\ref{iimu}B}\label{{iimu}B}
\\
\int \varTheta \dd \omega &\colon f\longmapsto \int\biggl(\int f\dd\vartheta_x\biggr) \dd \omega(x) \in \RR\cup -\infty \quad\text{for each }f\in \text{\sf usc}(O).
\tag{\ref{iimu}f}\label{{iimu}f}
\end{align}
\end{subequations}   

Let $r\in \RR^+\!\setminus\!0$ and $\vartheta\in \Meas^+_{\comp}(r \BB)$.
 For $x\in \RR^{\tt d}$,  we define   
the  {\it shift\/} $\vartheta_x\in  \Meas^+_{\comp}\bigl( B(x,r)\bigr)$ of this measure $\vartheta$ to point $x$ as
\begin{subequations}\label{shift}
\begin{align}
\vartheta_x(B)& :=\vartheta(B-x)\quad \text{for any $B\in \Borel\bigl( B(x,r)\bigr)$},
\tag{\ref{shift}B}\label{{shift}B}
\\
\int f \dd \vartheta_x&:=
\int_{r\BB} f(x+y) \dd \vartheta (y) \in \RR\cup -\infty \quad\text{for each }f\in \text{\sf usc}\bigl(B(x,r)\bigr).
\tag{\ref{shift}f}\label{{shift}f}
\end{align}
\end{subequations}
For a measure \eqref{omega}, under the condition 
\begin{equation}\label{Su}
S^{\cup r}:=\bigcup_{x\in S} B(x,r)\Subset O,
\end{equation}
we can  define the {\it convolution\/} $\omega *\vartheta\in \Meas^+_{\comp}(O)$ of measures $\omega$ and $\vartheta$ by the 
 the integral $\int \varTheta \dd \omega$ of $\varTheta\colon S\overset{\eqref{{sxi}t}}{\longrightarrow} \Meas^+_{\comp}(O)$ with respect to the measure $\omega$ as 
\begin{subequations}\label{o*}
\begin{align}
\omega *\vartheta&\overset{\eqref{{iimu}I}}{:=}\int \varTheta \dd \omega
\overset{\eqref{shift}}{=} \int_S \vartheta_x\dd\omega(x)\in \Meas^+_{\comp}(O),
\tag{\ref{o*}*}\label{{o*}*}\\
(\omega *\vartheta)(B)&\overset{\eqref{{shift}B}}{=}\int_S \vartheta (B-x)\dd\omega(x)\in \RR
\quad\text{for each $B\in \Borel (O)$ such that $B\Subset O$,}
\tag{\ref{o*}B}\label{{o*}B}
\\
\int f \dd (\omega *\vartheta)&\overset{\eqref{{shift}f}}{=}
\int_S\biggl(\int_{r \BB} f(x+y) \dd \vartheta (y)\biggr) 
\dd \omega (x)
\in \RR\cup -\infty
 \quad\text{for each }f\in \text{\sf usc}(O).
\tag{\ref{o*}f}\label{{o*}f}
\end{align}
\end{subequations}

Very special cases of the following Theorem \ref{pr:ii} were essentially used for    convolutions in \cite[Lemmata 7.1, 7.2]{Kha01}, \cite[2.1.1, 1b)]{BaiTalKha}, \cite[8.1]{KhaRozKha19}. 

\begin{theorem}\label{pr:ii} 
Let   $\omega\in \Meas_{\comp}(O)$   be a measure from \eqref{omega}.

If $\varnothing \neq H\subset \text{\sf usc}(O)$ and 
  each measure $\vartheta_x\overset{\eqref{{sxi}t}}{=}\varTheta(x)$
 in \eqref{sxi} is $H$-balayage of the Dirac measure $\delta_x$ at $x\in S$, then
the integral $\int \varTheta \dd \omega\overset{\eqref{iimu}}{\in} \Meas^+_{\comp}(O)$  is  $H$-balayage  of $\omega$, i.e.,  
\begin{equation}\label{Th}
\omega\overset{\eqref{balnumu}}{\preceq}_H\int \varTheta \dd \omega
\overset{\eqref{iimu}}{=}\int_S \vartheta_x\dd\omega(x)\in \Meas^+_{\comp}(O).
\end{equation}

If $H=\har (O)$  (respectively, $H=\sbh(O)$), $r\in \RR^+\!\setminus\!0$,  and 
a measure $\vartheta\in \Meas^+_{\comp}(r\BB)$
  is an Arens\,--\,Singer (respectively, a Jensen)
 measure on $r\BB$ at $0\in r\BB$, then, under condition  \eqref{Su}, 
the convolution  $\omega *\vartheta \overset{\eqref{o*}}{\in} \Meas^+_{\comp}(O)$  is  $\har(O)$(respectively, $\sbh(O)$)-balayage 
of $\omega$, i.e.,  
\begin{equation}\label{Th*}
\omega\preceq_{\har(O)}\Bigl(\text{respectively, }\preceq_{\sbh(O)} \Bigr)\omega*\vartheta
 \in \Meas^+_{\comp}(O).
\end{equation}
\end{theorem}
\begin{proof} Under conditions \eqref{sxi}--\eqref{omega},
by definition \eqref{iimu} and by Definition \ref{df:1} for  $H$-balayage $\delta_x\preceq_H\vartheta_x$,  for each function  $h\in H\subset {\sf usc}(O)$, we have
\begin{equation}\label{ins}
\int h\dd \omega = \int \int h\dd \delta_x\dd \omega \leq 
\int_S\biggl(\int_{\supp \vartheta_x}h \dd \vartheta_x\biggr)\dd \omega (x)
\overset{\eqref{{iimu}f}}{=}\int h \dd  \int \varTheta \dd \omega
\quad\text{for each $h\in H$.}  
\end{equation}
By Definition \ref{df:1}, the latter means \eqref{Th}. By definition \eqref{o*}
of convolution $\omega*\vartheta$, the final part of Theorem \ref{pr:ii}  with formula \eqref{Th*} is a special case of the proved part \eqref{Th} of Theorem \ref{pr:ii}.
\end{proof}

\section{Polar sets and balayage with an example} 
Remind that a set $E\subset \RR^{\tt d}$ is {\it polar\/}
 if there is  $u\in \sbh_*(\RR^{\tt d})$ such that $E\subset 
\bigl\{x\in \RR^{\tt d}\colon u(x)=-\infty\bigr\}$, or, in equivalent form,  
$\text{Cap}^*E=0$ if we use the {\it outer capacity\/} 
\begin{equation}\label{E}
\text{Cap}^*(E):=\inf_{E\subset O'\overset{\text{\tiny open}}{=}\Int O'}  
\sup_{\stackrel{C\overset{\text{\tiny closed}}{=}\clos C\overset{\text{\tiny compact}}{\Subset} O}{\nu\in \Meas^{1+}(C)}} 
 k_{d-2}^{-1}\left(\iint K_{d-2} (x,y)\dd \nu (x) \dd \nu(y) \right).
\end{equation}

\begin{theorem}\label{Pr_pol}
If a measure  $\omega \in \Meas_{\comp}^+(O)$ is $\sbh(O)$-balayage of a measure $\delta \in \Meas_{\comp}^+(O)$, i.e.,   $\omega\preceq_{\sbh(O)}\delta$, and $E\subset \RR^{\tt d}$ is polar, i.e., $\text{Cap}^*E\overset{\eqref{E}}{=}0$, then 
 $\omega (O\cap E\!\setminus\!\supp {\delta})=0$.
\end{theorem}
\begin{remark}
A special case of this Theorem \ref{Pr_pol}  is noted in \cite[Corollary 1.8]{C-R} for a Jensen measure  $\omega\in J_x(O)$ on $O$ at $x\in O$ and the Dirac measure $\delta:=\delta_x$. It  was used in \cite[Lemma 3.1]{Kha03}. 
\end{remark}

\begin{proof} There is $k_0\in \NN$ such that $B(x,1/k_0)\Subset O$ for all 
$x\in \supp {\delta}$.  For any $k\in k_0+\NN_0$ 
there exists an finite cover of $\supp {\delta}$ by balls $B(x_j,1/k)\Subset O$ such that the open subsets 
\begin{equation*}
O_k:=\bigcup_j B(x_j,1/k)\Subset O,\quad
 \supp {\delta} \Subset O_k\supset O_{k+1}, \quad k\in k_0 +\NN_0, \quad \supp {\delta} =\bigcap_{k\in k_0+\NN_0} O_k,  
\end{equation*}
have complements $\RR_{\infty}^{\tt d} \!\setminus\!O_k$ in $\RR_{\infty}^{\tt d}$ \textit{without isolated points.\/} Then 
 every open set  $O_k\Subset O$ is regular for the Dirichlet problem.
It suffices to prove that the equality $\omega (O_k\cap E)=0$ holds for every number  
 $k\in k_0+\NN_0$. By definition of polar sets, there is a  function $u\in \sbh_*(O)$ such that $u(E)=\{-\infty\}$. Consider the functions 
\begin{equation}\label{Uk}
U_k=\begin{cases}
u \text{ \it  on $O\!\setminus\!O_k$},\\
\text{\it the harmonic extension of $u$ from $\partial O_k$ into $O_k$}\text{ on $O_k$},
\end{cases}   
\qquad k\in k_0+\NN_0.
\end{equation}
We have  $U_k\in \sbh_*(O)$, and $U_k$ is bounded from below in $\supp {\delta} \Subset O_k$. Hence
\begin{multline*}\label{<U}
-\infty <\int_O U_k \dd {\delta}
\overset{\eqref{balnumu}}{\leq}
\int_O U_k \dd \omega=
\left(\int_{O\!\setminus\!(O_k\cap E)}+\int_{O_k\cap E}\right) U_k \dd \omega
\\
=\int_{O\!\setminus\!(O_k\cap E)} U_k \dd \omega+(-\infty)\cdot \omega(O_k\cap E)
\leq \omega(O) \sup_{\supp \omega} U_k+(-\infty)\cdot \omega(O_k\cap E).
\end{multline*}
Thus, we have  $\omega(O_k\cap E)=0$.
\end{proof}

Generally speaking, Theorem \ref{Pr_pol} is not true for $\har(O)$-balayage.  An implicit example is built in \cite[Example]{MenKha19}. 
We get in Example 5 another already constructive way to build such examples.

\begin{example}[{\rm development of one example  of T. Lyons \cite[XIB2]{Koosis}}] \label{5}
Let $\lambda$ be the Lebesgue measure on $\RR^{\tt d}$, and let ${\tt b}$ be the volume of the unit ball $\BB\subset \RR^{\tt d}$. Consider 
\begin{equation}\label{Eas}
O:=\BB,  \quad 0<t<r<1,  \quad \delta:=\frac1{{\tt b}t^{\tt d}}\lambda\bigm|_{t\BB}, 
\quad \omega :=\frac1{{\tt b}r^{\tt d}}\lambda\bigm|_{r\BB}.  
\end{equation} 
Easy to see that $\delta \preceq_{\sbh(\BB)} \omega$. Let $E=(e_j)_{j\in \NN}\Subset r\BB\!\setminus\!t\overline \BB$ be a polar countable set without limit point in $r\BB\!\setminus\!t\overline \BB$.  Surround each point $e_j\in E$ with a ball $B(e_j,r_j)$ of such a small radius  $r_j>0$ that the union of all these balls is contained in $r\BB\!\setminus\!t\overline \BB$. Consider a measure 
\begin{multline*}
\mu_E:=\omega-\frac{1}{{\tt b}r^{\tt d}}\sum_{j\in \NN} \lambda\bigm|_{B(e_j,r_j)}+\frac{1}{{\tt b}r^{\tt d}}\sum  \lambda(e_j,r_j)\delta_{e_j}
\\
\overset{\eqref{Eas}}{=}
\frac1{{\tt b}r^{\tt d}}\lambda\bigm|_{r\BB}-\frac{1}{{\tt b}r^{\tt d}}\sum_{j\in \NN} \lambda\bigm|_{B(e_j,r_j)}+\frac{1}{r^{\tt d}}\sum_{j\in\NN} r_j^{\tt d} \delta_{e_j} .
\end{multline*}
By construction, the measure $\mu_E$ is $\har(\BB)$-balayage of measure $\delta$, 
but 
\begin{equation*}
\mu_E(E)=\frac{1}{r^{\tt d}}\sum_{j\in \NN} r_j^{\tt d} >0
\end{equation*}
in direct contrast to Theorem \ref{Pr_pol}. 
\end{example}


\section{Balayage for three measures}

\begin{propos}\label{pr:4}  Let  $\omega \in \Meas_{\comp}(O)$ and  ${\delta}\in  \Meas_{\comp}(O)$. 

If  $\omega$ is $\sbh (O)$-balayage of  $\delta$, then 
\begin{equation}\label{bhs}
\int  u \dd {\delta}\leq \int  u\dd \omega \quad \text{ for each  $u\in \sbh (S_O)$},
\quad\text{where $S_O=\infill_O S$,  $S:=\supp \omega \cup \supp {\delta}$},
\end{equation}
i.\,e.,  if $O'\supset S_O$ is an open subset in $\RR^{\tt d}$, then 
$\omega$ is  $\sbh(O')$-balayage of $\delta$. 

If  $\omega$ is $\har (O)$-balayage of  $\delta$, then 
\begin{equation}\label{hs}
\int  h \dd {\delta}=\int  h\dd \omega \quad \text{ for each  $h\in \har (S_O)$},
\end{equation}
i.\,e.,   if $O'\supset S_O$ is an open subset in $\RR^{\tt d}$, then 
$\omega$ is  $\har(O')$-balayage of $\delta$. 
\end{propos}
\begin{proof} If $u\in \sbh (S_O)$, then, by Proposition \ref{prs+}, there is  a function $ U \in \sbh \bigl(O)$ such that $u=U$ on  $ S_O$ , and, in the case $\delta\preceq_{\sbh(O)}\omega$, we have
\begin{equation*}
\int_{S_O}  u \dd {\delta}= \int_{S_O}  
 U \dd {\delta} = \int_{O} U \dd {\delta}
\overset{\eqref{balnumu}}{\leq} 
 \int_{O} U \dd \omega= 
 \int_{S_O} U \dd \omega=
 \int_{S_O}  u \dd {\omega},
\end{equation*}
that   gives \eqref{bhs}. If $\delta\preceq_{\har(O)}\omega$, then we can repeat 
\eqref{ihSO} using Proposition \ref{pr3}, and we obtain \eqref{hs}. 
\end{proof}

Very special cases of the following Theorem \ref{pr:vstm} were essentially used in \cite[Proposition 3]{BaiTalKha}  only for special Jensen measures 
on the complex plane on the complex plane identified with $\RR^2$.

\begin{theorem}\label{pr:vstm} Suppose  that measures  $\beta, \delta, \omega \in \Meas_{\comp}^+(O)$ satisfy the conditions
\begin{equation}\label{var}
\begin{cases}
\beta &\preceq_{\har(O)}\delta,\\ 
\beta &\preceq_{\sbh(O)} \omega,
\end{cases}
\qquad \text{and }\infill (\supp \beta \cup \supp \delta )\subset O',
\end{equation}
where $O'\Subset O$ is an open subset such that $O'\cap \supp \omega=\varnothing$. Then
$\delta\preceq_{\sbh(O)} \omega$.	
\end{theorem}
\begin{proof}  It suffices to consider the case when $D:=O$ and $D':=O'$ are domains. There exists a regular (for the Dirichlet problem) domain $D''$ such that 
\begin{equation}\label{inss}
\infill (\supp \beta \cup \supp \delta ) \overset{\eqref{var}}{\subset} D''\Subset D'\subset D
\end{equation}
since $\infill (\supp \beta \cup \supp \delta )$ is compact subset in $D'$ by Proposition \ref{KOc}\eqref{Ki}.

Let $u\in \sbh_*(D)$. Then we can build a new subharmonic function $\widetilde u \in \sbh_*(D)$ such that 
\begin{equation}\label{tu}
\widetilde u\bigm|_{D''}\in \har (D''), \quad \widetilde u=u\quad\text{on $D\!\setminus\!D''$}, \quad 
u\leq \widetilde u\quad \text{on $D$}.
\end{equation} 
By Proposition \ref{pr:4}, in view of the inclusion in \eqref{inss}, we have
\begin{equation}\label{chain}
\int_{D} u\dd \delta 
\overset{\eqref{var}}{=}\int_{D''}  u\dd \delta
\overset{\eqref{tu}}{\leq}\int_{D''} \widetilde u\dd \delta
\overset{\eqref{inss},\eqref{hs}}{=} \int_{D''} \widetilde u\dd \beta=
\int_{D} \widetilde u\dd \beta
\overset{\eqref{var}}{\leq} \int_{D} \widetilde{u} \dd \omega.
\end{equation}
Since $\supp \omega \subset D\!\setminus\!D' $, we can continue this chain of (in)equalities \eqref{chain} as 
\begin{equation*}
\int_{D} u\dd \delta \overset{\eqref{chain}}{\leq} \int_{D} \widetilde{u} \dd \omega
=\int_{D\!\setminus\!D'} \widetilde{u} \dd \omega
\overset{\eqref{tu}}{=}\int_{D\!\setminus\!D'} u  \dd \omega=\int_{D} u  \dd \omega.
\end{equation*}
 This completes the proof.
\end{proof}

This research was supported by a grant of the Russian Science Foundation (Project No. 18-11-00002).






\end{document}